\documentclass[reqno,centertags,12pt]{amsart}
\usepackage{amsmath,amsthm,amscd,amssymb,latexsym,verbatim}
\usepackage{graphicx,hyperref}
\usepackage{newclude, amsfonts, mathrsfs}
\usepackage{graphicx,epsf,cite}
\usepackage{float}

-\marginparwidth \oddsidemargin -4mm \evensidemargin \oddsidemargin

\newtheorem{theorem}{Theorem}[section]

\newtheorem{proposition}[theorem]{Proposition}
\newtheorem{lemma}[theorem]{Lemma}
\newtheorem{corollary}[theorem]{Corollary}
\theoremstyle{definition}
\newtheorem{definition}[theorem]{Definition}

\theoremstyle{remark}
\newtheorem{remark}[theorem]{Remark}

\newcounter{smalllist}

\allowdisplaybreaks
\numberwithin{equation}{section}

\newcommand{\abs}[1]{\left|#1\right|}
\newcommand{\norm}[1]{\left| \left|#1\right| \right|}

\newcommand{\FullBurgers}[1]{\partderiv{#1}{t}(x,t) + \Big( {#1}(x+h,t) \pm #1 (x-h,t)\Big)\partderiv{#1}{x}(x,t) = 0}
\newcommand{\Burgers}[1]{\partderiv{#1}{t} + #1 \cdot \nabla{#1} = \gamma \Delta{#1}}

\newcommand{\InviscidBurgers}[1]{\partderiv{#1}{t} + #1 \cdot \partderiv{#1}{x} = 0}

\newcommand{\InviscidBurgersNonlocal}[1]{{#1}_{t} + \big({#1}(x+h,t) \pm {#1}(x-h,t) \big) {#1}_x = 0}
\newcommand{\InviscidBurgersNonlocalPlus}[1]{{#1}_{t} + \big({#1}(x+h,t) + {#1}(x-h,t) \big) {#1}_x = 0}
\newcommand{\InviscidBurgersNonlocalMinus}[1]{{#1}_{t} + \big({#1}(x+h,t) - {#1}(x-h,t) \big) {#1}_x = 0}
\newcommand{\recursiveBurgersNonlocal}[1]{\partial_{t} #1_{n} + \calL#1_{n-1} \partial_{x} #1_{n}= 0}

\newcommand{\beq}{\begin{equation}}
\newcommand{\eeq}{\end{equation}}

\newcommand{\bal}{\begin{align}}
\newcommand{\eal}{\end{align}}
\newcommand{\bals}{\begin{align*}}
\newcommand{\eals}{\end{align*}}

\newcommand{\tht}{\theta}

\newcommand{\ffi}{\varphi}

\newcommand{\partialt}{\partial_{t}}
\newcommand{\partialx}{\partial_{x}}
\newcommand{\deriv}[2]{\frac{\intd{#1}}{\intd{#2}}}

\newcommand{\partderiv}[2]{\frac{\partial{#1}}{\partial{#2}}}

\newcommand{\intd}{\mathrm{d}}

\newcommand{\bbR}{{\mathbb{R}}}

\newcommand{\bbZ}{{\mathbb{Z}}}

\newcommand{\calL}{{\mathcal L}}

\title{Local Well-Posedness of A Nonlocal Burgers Equation}

\author{Hang Yang}
\thanks{1. Undergraduate at University of Illinois - Urbana-Champaign. Contact: \href{mailto:hyang62@illinois.edu}{hyang62@illinois.edu}}
\author{Sam Goodchild}
\thanks{2. Undergraduate at University of Wisconsin - Madison. Contact: \href{mailto:sgoodchild11692@gmail.com}{sgoodchild11692@gmail.com}}

\begin{document}

\begin{abstract}
	In this paper, we explore a nonlocal inviscid Burgers equation. Fixing a parameter $h$, we prove existence and uniqueness of the local solution of the equation $\InviscidBurgersNonlocal{u}$ with periodic initial condition $u(x,0) = u_0(x)$. We also explore the blow up properties of solutions to these kinds of equations with given periodic initial data, and show that there exists solutions that blow up in finite time and solutions that are globally regular. This contrasts with the classical inviscid Burgers equation, for which all non-constant smooth periodic initial data lead to finite time blow up.
Finally, we present results of simulations to illustrate our findings.
\end{abstract}
\maketitle

\section{Introduction} \label{Introduction}

The Burgers equation is a common equation that arises naturally in the study of fluid mechanics, traffic, and other fields. It is a relatively simple partial differential equation that has been extensively studied. In finite time, solutions to the inviscid Burgers equation is known to devellop shock waves and rarefaction for any smooth initial data. It also serves as an basic example of conservation laws. Many different closed-form, series approximation, and numerical solutions are known for particular sets of boundary conditions.

The more general form of dissipative Burgers equation looks like \begin{align}  \Burgers{u},\label{Burgers} \end{align} where $u(x,t)$ represents the velocity at point $(x,t) \in \bbR \times \bbR^{+}$, $\gamma \in \bbR^+$, and the term on the right hand side is the viscosity term which induces diffusion properties. For the inviscid 1D case, Burgers equation reduces to \begin{align} \InviscidBurgers{u}. \label{InviscidBurgers} \end{align}

The equation that we will be studying is \begin{align} \FullBurgers{u} \label {FullBurgers} \end{align} As we can see in the equation, which is a generalized form of the trivial 1D Burgers' equation, it introduces nonlocal factors. Unlike the local Burgers equation, analytical solutions are extremely hard to discover for this kind of equation. Even so, existence of solution cannot be easily derived from the method of characteristics also. If we look at the characteristics, which are defined by $dx/dt = u(x+h,t) \pm u(x-h,t)$, they are hard to analyze due to the nonlocality. 

In Section \ref{Existence and Uniqueness Of Solution}, we prove the following two theorems, illustrating the existence and uniqueness of classical local solutions for periodic initial data $u(x,0) = u_0(x)$. 

\begin{theorem}[Local Existence] \label {LocalExistence}
	Suppose $u_0 \in C^\infty(\bbR)$ with period $L$. Then there exists a classical local solution $u(x,t)$ to \eqref{BurgersNonlocal} for $0 \leq t \leq T(u_0)$, for some $T(u_0)>0$.
\end{theorem}

\begin{theorem}[Uniqueness] \label{uniqueness} \label {uniqueness}
The solution $u(x,t)$  to equation \eqref{BurgersNonlocal} which is in $C^1([0,T], H^r)$ for large enough $r$ is unique.
\end{theorem}

We resort to functional analysis skills in Sobolev spaces. Basically, we use the original equation to generate a revursive sequence of functions and prove that in appropriately chosen Sobolev spaces, the sequence admits a unique limit that converges to a classical local solution. In Section \ref{Blow Up and Non Blow Up Properties} we look at blow up and non blow up of solutions in finite time, presenting examples of both cases and contrast with the local Burgers equation.Interestingly, owing to the nonlocality factors introduced,the blow up behavior of \eqref {FullBurgers} may contrast greatly with the normal Burgers equation \eqref{InviscidBurgers}.Finally, we use graphics to show simulations run on our equation in Section \ref{Simulations} to illustrate our results.

\section{Existence and Uniqueness Of Solution} \label{Existence and Uniqueness Of Solution}
Let us now consider the following nonlocal variation of Burgers equation:
\begin{eqnarray}
	\InviscidBurgersNonlocal{u}. \label{BurgersNonlocal}
\end{eqnarray}
In the following theorem, we prove that there exists a solution to this equation. We construct a sequence of functions $u_n(x,t)$ and show that $u_n(x,t)$ will be uniformly bounded in $C([0,T],H^m)$ with large $m$,
while $du_n/dt$ are also uniformly controlled. Thus, by well known compactness criteria, there exists a limit which we show solves the equation. Let us first define Sobolev spaces as follows:
\begin{definition}[Sobolev Norm]
	Let $u(x,t) \in C^{\infty}(\bbR)$ be periodic with period $L$, for some $m \in \bbZ^+$. Then the Sobolev norm is defined as
	\begin{align*}
		\norm{u (\cdot, t)}_{H^m}^2 &= \int\limits_{0}^{L} u(x,t) \left[(-\partial_{xx} )^m u(x,t) \right] \intd x \\&=  \int\limits_{0}^{L} \abs{\partialx^m u(x,t)}^2 \intd x
	\end{align*}
\end{definition}
\begin{remark}
	The Sobolev space $H^{m}([0,L])$ is the closure of $C^{\infty}([0,L])$ with respect to this norm. Observe that we will work with what is usually called the homogeneous Sobolev space $\dot{H}^m$.
\end{remark}

	We will prove Theorem \ref {LocalExistence} by proving Proposition \ref{convergentsubsequence} and Lemma \ref{existence} below.
	\begin{remark}
Throughout the rest of the paper, we will denote any universal constant by $C$, which does not depend on $u$ and may vary from line to line.
\end{remark}

\begin{proposition} \label{convergentsubsequence}
	Define a recursive sequence of functions $\{ u_n \}$ as follows:
\begin{eqnarray}
	\recursiveBurgersNonlocal{u}, & u_n(x,0) = u_0(x)\label{recursiveBurgersNonlocal}
\end{eqnarray}	
where $u_{n} = u_{n}(x,t)$ for $n\geq 1$, $\calL u_{n} = u_{n}(x+h,t) \pm u_{n}(x-h,t)$ is a short-hand notation,  and $u_0(x,t)= u_0 (x)$ is smooth. Then for all sufficiently large $m \in \bbZ^+$, there exists $T(\|u_0\|_{H^m})$ such that $\norm{u_n(\cdot,t)}_{C([0,T],H^m)} < C_1(T)$ and $\|du_n/dt\|_{C([0,T],H^{m-1})} \leq C_2(T)$, for all $0<t<T$. Moreover, there exists a subsequence $n_j$ such that $u_{n_j}(x,t)$ to $u(x,t)$ in $C([0,T],H^r)$ for any $r<m$. 
\end{proposition}

\begin{proof}
By method of characteristics, since $u_{n-1}(x,t)$ is a known smooth function, each $u_n(x,t)$ exists for all time. Also note that $u_n(x,t)$ has period $L$. This follows from an inductive argument because $\calL u_{0}$ has period $L$.  Let us multiply equation \eqref{recursiveBurgersNonlocal} by $\left(-1\right)^m \partialx^{2m} u_{n}$ to get \begin{eqnarray*} \left(-1\right)^m  \partialt u_{n} \partialx^{2m} u_{n} + \left(-1\right)^m \partialx^{2m} u_{n} \calL u_{n-1} \partialx u_{n} = 0. \end{eqnarray*}
Simplifying, we get \begin{eqnarray*} \partialt u_{n} \partialx^{2m} u_{n} = - \partialx^{2m} u_{n}  \calL u_{n-1} \partialx u_{n}. \end{eqnarray*}
Integrating this expression with respect to $x$ from $0$ to $L$ yields \begin{eqnarray*}  \int\limits_{0}^{L} \partialt u_{n} \partialx^{2m} u_{n} \intd x = - \int\limits_{0}^{L} \partialx^{2m} u_{n}  \calL u_{n-1} \partialx u_{n} \intd x. \end{eqnarray*}
We can then integrate by parts $m$ times and pull out the partial derivative with respect to time from the left hand side, giving \begin{eqnarray*} \deriv{}{t} \int\limits_{0}^{L} \left( \partialx^m u_n \right)^2 \intd x = 2\left(-1\right)^{m+1} \int\limits_{0}^{L} \partialx^{2m} u_{n} \calL u_{n-1} \partialx u_{n} \intd x.\end{eqnarray*}
The integral on the left hand side, by definition, is equal to $\norm{u_n (\cdot, t)}_{H^m ([0,L])}^2$, so we can write \begin{eqnarray*} \deriv{}{t} \norm{u_n (\cdot, t)}_{H^m}^2 \leq 2\abs{\int\limits_{0}^{L} \partialx^{2m} u_{n} \calL u_{n-1} \partialx u_{n} \intd x}. \end{eqnarray*}
Integrating by parts $m$ times and noting that all of the boundary terms vanish due to periodicity, we get \begin{align} \deriv{}{t} \norm{u_n (\cdot, t)}_{H^m}^2 \nonumber &\leq \abs{\int\limits_{0}^{L} \partialx^m \big( \calL u_{n-1} \partialx u_n\big) \partialx^m u_n \intd x} \\ &\leq \abs{\int\limits_{0}^{L} \sum\limits_{l=0}^m \left( \begin{array}{c} m \\ l \end{array} \right)\partialx^l \big( \calL u_{n-1} \big) \partialx^{m-l+1} u_n \partialx^m u_n \intd x} \nonumber \\ &\leq \sum\limits_{l=0}^m  \left( \begin{array}{c} m \\ l \end{array} \right)\abs{\int\limits_{0}^{L}  \partialx^l \big( \calL u_{n-1} \big) \partialx^{m-l+1} u_n \partialx^m u_n \intd x} \label{sumintegral} \end{align}

\begin{lemma} For all $0 \leq l \leq m$, \begin{align*} \abs{\int\limits_{0}^{L}  \partialx^l \big( \calL u_{n-1} \big) \partialx^{m-l+1} u_n \partialx^m u_n \intd x}  \leq C \norm{u_{n-1}}_{H^m} \norm{u_n}_{H^m}^2. \end{align*}
\end{lemma}
\begin{proof}
For the $l=0$ case, we can reduce this to the $l=1$ case using integration by parts:
\begin{align*}
	\abs{\int\limits_{0}^{L} \calL u_{n-1} \partialx^{m+1} u_n \partialx^m u_n \intd x} = C\abs{\int\limits_{0}^{L} \partialx \left(\calL u_{n-1}\right) \left(\partialx^m u_n\right)^2 \intd x}.
\end{align*}
When $l=1$ we get \begin{align*} \abs{\int\limits_{0}^{L} \partialx \left(\calL u_{n-1}\right) \left(\partialx^m u_n\right)^2 \intd x} & \leq \norm{\partialx \left( \calL u_{n-1}\right)}_{L^\infty} \cdot \abs{\int\limits_{0}^{L} \left(\partialx^m u_n\right)^2 \intd x}\\ & \leq \norm{\partialx \left( \calL u_{n-1}\right)}_{L^\infty} \cdot \int\limits_{0}^{L} \abs{\partialx^m u_n}^2 \intd x. \\ &= C \norm{\partialx \left( \calL u_{n-1}\right)}_{L^\infty} \cdot \norm{u_n}_{H^m}^2. \end{align*}
Applying the Sobolev embedding theorem, we have that for $m > 3/2$,  \begin{align*} \norm{\partialx \left(\calL u_{n-1}\right)}_{L^\infty} \leq C \norm{\partialx \left(\calL u_{n-1}\right)}_{H^{m-1}} \leq C \norm{\calL u_{n-1}}_{H^m}, \end{align*} and \begin{align*} \norm{\calL u_{n-1}}_{H^m} &= \norm{u_{n-1}(x+h,t) \pm u_{n-1}(x-h,t)}_{H^m} \\ &\leq 2 \norm{u_{n-1}}_{H^m}. \end{align*}
We can conclude \begin{align*} \abs{\int\limits_{0}^{L} \partialx \left(\calL u_{n-1}\right) \left(\partialx^m u_n\right)^2 \intd x} \leq C \cdot \norm{u_{n-1}}_{H^m} \cdot \norm{u_n}_{H^m}^2 \text{\indent for } m > \frac{3}{2}. \end{align*}
In general, by H\"older's inequality\nocite{McOwenPDE}, terms on the right hand side of \eqref{sumintegral}, for $l\neq 1$, are estimated by \begin{align} \abs{\int\limits_{0}^{L} \partialx^l \left(\calL u_{n-1} \right) \partialx^{m-l+1} u_n \partialx^m u_n \intd x} \leq  \norm{\partialx^l \left(\calL u_{n-1}\right)}_{L^{\frac{2(m-1)}{l-1}}}\cdot \norm{\partialx^{m-l+1} u_n}_{L^{\frac{2(m-1)}{m-l}}}\cdot \norm{\partialx^m u_n}_{L^2} \label{star star}\end{align}
Recall that Gagliardo-Nirenberg inequality (see e.g. \cite{DoeringNavierStokes})
has the form
\begin{align} \norm{\partialx^s f}_{L^{2m/s}} \leq C \norm{f}_{L^\infty}^{1-s/m} \norm{\partialx^m f}_{H^m}^{s/m} \indent \forall 1 \leq s \leq m. \label{Gagliardo}
\end{align}

Now by applying \eqref{Gagliardo} and the Sobolev embedding theorem, we can conclude the following two facts: \begin{align} \norm{\partialx^{m-l+1}u_n}_{L^{\frac{2(m-1)}{m-l}}} &= \norm{\partialx^{m-l}\left( \partialx u_n\right)}_{L^{\frac{2(m-1)}{m-l}}} \nonumber \\ &\leq C \cdot \norm{\partialx u_n}_{L^\infty}^{1-\frac{m-1}{m-l}}\cdot \norm{\partialx^m u_n}_{L^2}^{\frac{m-1}{m-l}} \nonumber \\ &\leq C \cdot \norm{\partialx^m u_n}_{H^{m-1}}^{1-\frac{m-1}{m-l}} \cdot \norm{\partialx u_n}_{H^{m-1}}^{\frac{m-1}{m-l}} \nonumber \\ &= C \cdot \norm{\partialx u_n}_{H^{m-1}} \nonumber \\ &= C \cdot \norm{u_n}_{H^m},   \label{b} \\  \nonumber \end{align}
\begin{align} \norm{\partialx^l \left(\calL u_{n-1}\right)}_{L^{\frac{2(m-1)}{l-1}}} &= \norm{\partialx^{l-1}\left[ \partialx \left( \calL u_{n-1}\right)\right] }_{L^{\frac{2(m-1)}{l-1}}} \nonumber \\ &\leq C \cdot \norm{\partialx \left(\calL u_{n-1} \right)}_{L^\infty}^{1-\frac{m-1}{l-1}} \cdot \norm{\partialx^m \left(\calL u_{n-1}\right)}_{L^2}^{\frac{m-1}{l-1}} \nonumber \\ &\leq C \cdot \norm{\partialx \left(\calL u_{n-1}\right)}_{H^{m-1}}^{1-\frac{m-1}{l-1}} \cdot \norm{\partialx^m \left(\calL u_{n-1}\right)}_{H^{m-1}}^{\frac{m-1}{l-1}} \nonumber \\ &= C \cdot \norm{\partialx \left(\calL u_{n-1}\right)}_{H^{m-1}} \nonumber \\ &= C \cdot \norm{\calL u_{n-1}}_{H^m} \nonumber \\ &\leq C \cdot \norm{u_{n-1}}_{H^m}. \label{c}\end{align}
Plugging \eqref{b} and \eqref{c} into \eqref{star star}, we get \begin{align*} \abs{\int\limits_{0}^{L} \partialx^l \left(\calL u_{n-1} \right) \partialx^{m-l+1} u_n \partialx^m u_n \intd x} &\leq C \norm{u_{n-1}}_{H^m} \cdot \norm{u_n}_{H^m}^2, \end{align*} with constant $C$ which depends only on $m$, so we have proved the lemma.
\end{proof}

Now let \begin{align*} f_0 (t) = f_n(0) = \norm{u_0(\cdot)}_{H^m}^2. \end{align*}
Notice, by definition, \begin{align*} \norm{u_n (\cdot,0)}_{H^m} = \norm{u_0(\cdot)}_{H^m}. \end{align*}
If we define $f_n(t)$ inductively by \begin{align*} f_n ' (t) = C(m) \sqrt{f_{n-1}(t)}f_n(t), \end{align*} then, given \begin{align*} \deriv{}{t} \norm{u_n(\cdot,t)}_{H^m}^2 \leq C(m) \norm{u_{n-1}(\cdot,t)}_{H^m} \cdot \norm{u_n(\cdot,t)}_{H^m}^2, \end{align*} we conclude \begin{align*} f_n(t) \geq \norm{u_n(\cdot,t)}_{H^m}^2. \end{align*}
On the other hand, it is easy to see by induction that \begin{align*} f_{n-1}(t) \leq f_n(t)\end{align*}
Thus \begin{align*} f_n ' (t) = C \sqrt{f_{n-1} (t)} f_n (t) \leq C(m) f_{n}^{3/2} (t). \end{align*}

Because $f_n(t) \neq 0$, we can divide by $f_n^{3/2}(t)$ to get \begin{align*} \frac{f_n ' (t)}{f_n^{3/2}(t)} \leq C. \end{align*}
We can then integrate from $0$ to $t$, giving \begin{align*} \int\limits_{0}^{t} \frac{f_n ' (s)}{f_n^{3/2}(s)} \intd s &\leq \int\limits_{0}^{t}  C \, \intd t \\ -2 f_n^{-1/2} (t) + 2\left(\norm{u_0 (\cdot)}_{H^m}\right)^{-1/2}&\leq  C t \\ f_n^{1/2} (t)&\leq \frac{1}{\norm{u_0(\cdot)}_{H^m}^{-1/2} -  C t/2}. \end{align*}

If we let $T:= \left( C\sqrt{\norm{u_0 (\cdot)}_{H^m}}\right)^{-1}$, we can conclude that for any $0 \leq t \leq T$, $\{f_n (t) \}$ will be uniformly bounded by some constant $C(T)$. Since $\{f_n(t)\}$ is monotonically increasing, $f_n (t) \rightarrow f(t)$. But $f_n(t) \geq \norm{u_n (\cdot,t)}_{H^m}^2$ therefore
\begin{align*}
	\sup\limits_{t \in [0,T]} \norm{u_n(\cdot , t)}_{H^m(\bbR)} \leq C(T).
\end{align*}
Since $u_n$ satisfies \eqref{recursiveBurgersNonlocal}, and $H^s$ in dimension one is an algebra for every $s>1/2$, this bound implies also
\begin{align*}
	\norm{\deriv{u_n}{t}(\cdot,t)}_{H^{m-1}[0,L]} \leq C(T),
\end{align*}
if $m$ is large enough. By the well known compactness criteria, see e.g. \cite{ConstantinNavierStokes}, it follows that there is a subsequence $u_{n_j}(x,t)$ that converges to $u(x,t)$ in the $C([0,T],H^r)$ for any $r \leq m$.
\end{proof}

\begin{lemma} \label{existence}
	The solution $u(x,t)$ from Lemma \ref{convergentsubsequence} solves equation \eqref{BurgersNonlocal}, and belongs to $C^1([0,T], H^r)$ for all $r<m-1.$
\end{lemma}

\begin{proof}
Pick $m$ large enough; $m >7/2$ is sufficient for the argument below to work.
	We have the recursive formula for $u_n$ in equation \eqref{recursiveBurgersNonlocal} and we proved in Theorem \ref{convergentsubsequence} that $u_n$ converges to $u$ in $C([0,T],H^r)$ for every
$r<m-1.$ Then for some $s>3/2$ we have \begin{align*} \norm{\calL u_{n-1}\partialx u_n - \calL u \partialx u}_{H^s} &\leq \norm{\left( \calL u_{n-1}-\calL u \right) \partialx u_n}_{H^s} + \norm{\calL u \left( \partialx u_n - \partialx u \right)}_{H^s} \\ &\leq \norm{\calL u_{n-1} - \calL u}_{H^s} \norm{\partialx u_n }_{H^s} + \norm{\calL u}_{H^s} \norm{\partialx u_n - \partialx u }_{H^s} \\ &\leq C \norm{u_{n-1} - u_n}_{H^s} \norm{\partialx u_n}_{H^s} + C \norm{\calL u}_{H^s}\norm{ u_n - u }_{H^{s+1}}. \end{align*}
	By our choice of $m,r$ and $s$, we have \begin{align*} \norm{u_{n-1} - u_n}_{H^s} \rightarrow 0 \text{ uniformly in $x$ as } n \rightarrow \infty \\ \norm{u_n - u }_{H^{s+1}} \rightarrow 0 \text{ uniformly in $x$ as } n \rightarrow \infty. \end{align*}
	Thus \begin{align*} \norm{\calL u_{n-1}\partialx u_n - \calL u \partialx u}_{H^s} \rightarrow 0 \text{ uniformly in $x$ as } n \rightarrow \infty. \end{align*}
	Now, integrating \eqref{recursiveBurgersNonlocal} from $0$ to $t$, we have \begin{align} u_n(x,t) = u_n(x,0) - \int\limits_{0}^t \calL u_{n-1} \, \partialx u_n \, \intd s = u_0(x) - \int\limits_{0}^t \calL u_{n-1} \, \partialx u_n \, \intd s. \label{integral in proof} \end{align}
	Since by our choice of $r,$ $u_n \rightarrow u$ pointwise from \eqref{integral in proof} we conclude
	\begin{align*}
		u(x,t) = u_0(x) - \int\limits_{0}^t \calL u \, \partialx u \, \intd s
	\end{align*}
	Differentiation with respect to $t$ and simple estimate finishes the proof of the lemma.
\end{proof}
We have therefore proved that there exists a solution to our equation, \eqref{BurgersNonlocal}. We now prove uniqueness by considering two different solutions of our equation, $\tht(x,t)$ and $\ffi(x,t)$, and showing that their difference $w(x,t) = \tht(x,t) - \ffi(x,t) =0$ for all $t$ and $x$.

Next, we prove that the classical solution is also unique, which is indicated in Theorem \ref {uniqueness}.
\begin{proof} 
	Suppose $\tht$ and $\ffi$ are both solutions to equation \eqref{BurgersNonlocal} with initial data $u(x,0) = u_0(x)$. Then
	\begin{align}
		\tht_t + \calL \tht \tht_x &= 0 \label{solution1} \\
		\ffi_t + \calL \ffi \ffi_x &= 0. \label{solution2}
	\end{align}
	Let $w = \tht - \ffi$. Subtracting \eqref{solution2} from \eqref{solution1}, we get
	\begin{align*}
		\partialt w &= - \left( \calL \tht \tht_x - \calL \ffi \ffi_x \right) \\
		&=- \left( \calL \tht \tht_x - \calL \ffi \ffi_x \right)+\calL \tht \ffi_x - \calL \tht \ffi_x. \\
		&=- \calL \tht w_x - \calL w \ffi_x.
	\end{align*}
	We multiply by $(-1)^r \partialx^{2r} w$, integrate from $0$ to $L$, and integrate the left hand side by parts $r$ times, giving
	\begin{align}
		\deriv{}{t} \int\limits_{0}^{L}  \left( \partialx^r w \right)^2 \intd x &= (-1)^{r+1} \int\limits_{0}^{L} \partialx^{2r} w  \calL \tht \partialx w \intd x + (-1)^{r+1}  \int\limits_{0}^{L} \partialx^{2r} w\calL w \partialx \ffi \intd x, \,\,{\rm so} \nonumber \\
		\deriv{}{t} \norm{w}_{H^r}^2 &\leq \underbrace{\abs{\int\limits_{0}^{L} \partialx^{2r} w \calL \tht \partialx w \intd x}}\limits_{I_1} + \underbrace{\abs{\int\limits_{0}^{L} \partialx^{2r} w \calL w \partialx \ffi \intd x}}\limits_{I_2} \label{star star star}.
	\end{align}
	Integrating $I_1$ by parts $r$ times gives
	\begin{align*}
		\abs{\int\limits_{0}^{L} \partialx^{2r} w \calL \tht \partialx w \intd x} \leq \sum\limits_{l=0}^r \left(\begin{array}{c} m \\l \end{array}\right) \abs{\int\limits_{0}^{L}  \partialx^l \big( \calL \tht \big) \partialx^{r-l+1} w\partialx^r w \intd x}.
	\end{align*}
	Again, when $l=0$, we can reduce this to the $l=1$ case using integration by parts. When $l=1$, \begin{align*} I_1 = \abs{\int\limits_{0}^{L} \partialx^l \big( \calL \tht \big) \partialx^{r-l+1} w\partialx^r w \intd x} &= \abs{\int\limits_{0}^{L} \partialx \big( \calL \tht \big) \partialx^{r} w\partialx^r w \intd x} \\
	&= \abs{\int\limits_{0}^{L} \partialx \big( \calL \tht \big) \left( \partialx^{r} w\right)^2 \intd x} \\ & \leq \norm{\partialx \left( \calL \tht \right)}_{L^\infty} \cdot \int\limits_{0}^{L} \abs{\partialx^r w}^2 \intd x \\
	&\leq C \cdot \norm{\partialx \left( \calL \tht \right)}_{L^\infty} \cdot \norm{w}_{H^r}^2 \\
	&\leq C \cdot \norm{\tht}_{H^r} \cdot \norm{w}_{H^r}^2,
	\end{align*}
if $r-1> 1/2.$
When $l \neq 1$,
\begin{align*}
	I_1 = \abs{\int\limits_{0}^{L} \partialx^l \big( \calL \tht \big) \partialx^{r-l+1} w\partialx^r w \intd x} &\leq \norm{\partialx^l \left(\calL \tht \right)}_{L^{\frac{2(r-1)}{l-1}}}\cdot \norm{\partialx^{r-l+1} w}_{L^{\frac{2(r-1)}{r-l}}}\cdot \norm{\partialx^r w}_{L^2} \\
	& \leq C \norm{\tht}_{H^r} \cdot \norm{w}_{H^r}^2.
\end{align*}
as before.
We can therefore conclude that
\begin{align*}
	I_1 = \abs{\int\limits_{0}^{L} \partialx^{2r} w \calL \tht \partialx w \intd x} \leq C \norm{\tht}_{H^r} \cdot \norm{w}_{H^r}^2.
\end{align*}
The same process can be done to $I_2$ to determine a bound for the integral, giving the result
\begin{align*}
	I_2 = \abs{\int\limits_{0}^{L} \partialx^{2r} w \calL w \partialx \ffi \intd x} \leq C \norm{\ffi}_{H^r} \cdot \norm{w}_{H^r}^2.
\end{align*}
Thus, \eqref{star star star} becomes
\begin{align*}
\deriv{}{t} \norm{w}_{H^r}^2 &\leq C \norm{\tht}_{H^r} \cdot \norm{w}_{H^r}^2 + C \norm{\ffi}_{H^r} \cdot \norm{w}_{H^r}^2 \\
&= \norm{w}_{H^r}^2 \Big( C \norm{\tht}_{H^r} + C \norm{\ffi}_{H^r} \Big).
\end{align*}
Then by Gronwall's inequality, we have
\begin{align*}
	\norm{w(\cdot,t)}_{H^r} \leq \norm{w(\cdot,0)}_{H^r}\exp{\left(\int\limits_{0}^{t} \big( C \norm{\tht(\cdot,s)}_{H^r} + C \norm{\ffi(\cdot,s)}_{H^r} \big)\intd s\right)},
\end{align*}
but $\norm{w(\cdot,0)}_{H^r}=0$ because $\tht$ and $\ffi$ are solutions to the same Cauchy problem. Therefore, the difference $w =\tht-\ffi= 0$ a.e. Since $\tht$ and $\ffi$ are sufficiently smooth, they must be equal everywhere.
\end{proof}

\section{Blow Up and Non Blow Up Properties} \label{Blow Up and Non Blow Up Properties}
	Let us consider the following two subcases of equation $\eqref{BurgersNonlocal}$, where they both have initial data $u_0(x)$ of period $L$:
	\begin{align}
		\InviscidBurgersNonlocalPlus{u} \label{plus} \\
		\InviscidBurgersNonlocalMinus{u} \label{minus}
	\end{align}
	\begin{remark} \label{evenness and oddness conservation}
	Let us introduce the following notation: denote $u^h(x,t)$ to be the solution of an equation with spatial shift $h$. 	Looking at equation \eqref{minus}, it is not difficult to show using symmetry and uniqueness that if the smooth initial condition $u_0 (x)$ is even, the solution, while it remains smooth, will stay even in $x$. Also, $u^h(x,t) = u^{L-h}(x,t)$ for all periodic initial data. Now consider equation \eqref{plus}. If $u_0 (x)$ is odd, the solution will stay odd in $x$. Also, $u^h(x,t) = u^{L-h}(x,t),$ will hold for all even initial data $u_0 (x)$.
	\end{remark}
	
	These facts are easily seen from the existence and uniqueness of solutions,  definitions of evenness and oddness, and periodicity applied to our equation.

We now state the existence of solutions that blow up in finite time.
\begin{theorem}[Existence of Blow Up]
	There exists initial data $u_0(x) \in C^\infty(\bbR)$ such that the solution $u(x,t)$ to \eqref{BurgersNonlocal} blows up in finite time.
\end{theorem}
We prove this result in Section \ref{Blow Up}. We first derive some properties of the solution.

\begin{lemma} \label{boundarylemma}
	Suppose $u(x,t)$ is a periodic solution of \eqref{BurgersNonlocal} with period $L=2h$. Let $u(0,0) = u(h,0) = 0$, then $u(0,t) = u(h,t)= 0,$ for all $t > 0$.
\end{lemma}
We can prove this by considering both the plus and minus cases as follows:
\begin{proof}
Let us first consider the plus sign case, \eqref{plus}. Plugging $x=0,h$ into to the recursive formula \eqref{recursiveBurgersNonlocal} for the plus case, we get
\begin{eqnarray*}
	&\partialt u_n (0,t) = -2u_{n-1}(h,t)\partialx u_n (0,t) \\
	&\partialt u_n (h,t) = -2u_{n-1}(0,t)\partialx u_n (h,t).
\end{eqnarray*}
Since $u(0,0) = u_0 (0) = u(h,0) = u_0(h) = 0$, we easily see $\partialt u_1 (0,t) = \partialt u_1 (h,t) = 0$, therefore $u_1$ is constant at $x=0,h$. But $u_1(0,0) = u_0(0) = 0$ and $u_1 (h,0) = u_0 (h) = 0$, so we have $u_1(0,t) = u_1(h,t) = 0$. Then, inductively, assume $u_{n-1}(0,t) = u_{n-1}(h,t) = 0$. Then, $\partialt u_n (0,t) = \partialt u_n (h,t) = 0$ so they are both constant. By the same reasoning, $u_n(0,0) = u_n(h,0) = 0$, therefore they are identically zero for all time. But our solution is just the limit of a subsequence of $u_n$, so $u(0,t) = u(h,t) = 0$

Now let us consider the minus sign case, \eqref{minus}. Plugging $x = 0$ into \eqref{minus}, we get
\begin{eqnarray*}
	u_t (0,t) = \Big[ u(h,t) - u(h,t)\Big]u_x (0,t) = 0,
\end{eqnarray*}
because $u(-h,t) = u(h,t)$ due to the period $L = 2h$. So $u(0,t) = C$, independent of time. Therefore, if we choose $u(0,0) = 0$, then $u(0,t) = 0$ for all $t \textgreater 0$. The same may be done at $u(h,0)$ to show that if $u(h,0) = 0$, then $u(h,t) = 0$.
\end{proof}

\begin{corollary} \label{all fixed stay fixed}
	Suppose $u_0(x)\in C^\infty(\bbR)$ has period $L=kh$ for some $k \in \bbZ$ and $u_0(mh) = 0$ for all $0 \leq m \leq k$. Then the solution to \eqref{BurgersNonlocal} satisfies $u(mh,t) = 0$ for all $t \geq 0$ and $0 \leq m \leq k$.
\end{corollary}
The proof is similar to that from Lemma \ref{boundarylemma} extended for more general integers.

\subsection{Blow Up} \label{Blow Up}		
	Now we investigate the cases where $u_0(x)$ has period $L = 2h$ and $u_0(0) = u_0(h) = 0$, and derive the possibility of blow up.
\begin{lemma} \label{finite time blowup lemma}
	Consider the equation $\InviscidBurgersNonlocalPlus{u}$ with $u(x,0) = u_0(x) \in C^\infty(\bbR)$, period $L=2h$, and $u_0(0) = u_0(h) = 0$. Assume $u_x(0,0) < 0$ and $u_x(h,0) < 0$. Then the solution $u(x,t)$ blows up in finite time.
\end{lemma}

\begin{proof}
	Note that in Theorem \ref{convergentsubsequence}, we proved that if the initial data $u_0(x)$ has period $2h$, then $u(x,t)$ will also have period $L = 2h$. Also, in this case, by Lemma \ref{plus}, $u(0,t) = u(h,t) = 0$.
	
	Differentiating the equation with respect to $x$ gives
	\begin{align}
		u_{tx} (x,t) &+ \left[ u_x(x+h,t) + u_x (x-h,t) \right] u_x(x,t) \nonumber \\ &+ \left[ u(x+h,t) + u(x-h,t)\right] u_{xx}(x,t) = 0.
	\end{align}
	Plugging in $x=0,h$ and noting that the last term vanishes, we get
	\begin{align*}
		u_{tx} (0,t) + 2u_x(h,t)u_x(0,t) &= 0 \\
		u_{tx} (h,t) + 2u_x(0,t)u_x(h,t) &= 0.
	\end{align*}
	Define $F_1(t) = u_x (0,t)$ and $F_2(t) = u_x(h,t)$. This gives the system
	\begin{align}
		F_1'+2F_1 F_2 = 0 \label{first ode} \\
		F_2'+2F_1 F_2 = 0. \label{second ode}
	\end{align}
	It is easy to see that $F_1 ' - F_2 ' = 0$, thus $F_1 - F_2 = A$, where $A$ is a constant. Since we assume $F_1 = F_2,$ we get that $A=0$. Plugging this in to \eqref{first ode} gives
	\begin{align*}
		F_1' + 2F_1^2 = 0.
	\end{align*}
	The solution to this differential equation is
	\begin{align*}
		F_1(t) = \frac{1}{\frac{1}{F_1(0)}+2t}.
	\end{align*}
	This blows up in finite time when $t = -\frac{1}{2F_1(0)} = -\frac{1}{2u_x(0,0)} > 0$. We can argue similarly for \eqref{second ode} to show that $F_2$ also blows up in finite time under the same conditions.
\end{proof}

\begin{remark}
	For instance, we can take $u(x,0) = u_0(x) = x(x-h)(x-2h)(-\frac{1}{2h^2} + \frac{3}{h^3}x - \frac{3}{2h^4}x^2)$, for $0 \leq x \leq 2h$. This satisfies our assumptions in Lemma \ref{finite time blowup lemma} and thus the
corresponding solution blows up in finite time.
\end{remark}

\begin{remark}
There is an obvious case of blow up for the plus sign equation when the period $L$ is just $h$. Equation \eqref{plus} reduces to \begin{align*} u_t + 2u \cdot u_x &= 0 \label{BurgersSimplified}. \end{align*} This is the typical Burgers equation, which is known to blow up in finite time for any non constant periodic initial condition $u_0 (x)$\cite{McOwenPDE}.
\end{remark}

\begin{lemma} \label{specific IC}
	Suppose $u_0$ has period $L = 6h$ and is even, and $u_0(kh) = 0$, $u_0'(3kh) = 0$ for all $k \in \bbZ$. Assume $\frac{\ln u_x(h,0) - \ln(-u_x(2h,0))}{u_x(h,0)+u_x(2h,0)} > 0$ and $ u_x(2h,0) < 0$. Then the solution $u(x,t)$ to the Cauchy problem,
	\begin{center}
		$\InviscidBurgersNonlocalMinus{u}$ \\
		$u(x,0) = u_0(x),$
	\end{center}
	blows up in finite time.
\end{lemma}

\begin{proof}
	By Lemma \ref{boundarylemma} and Corollary \ref{all fixed stay fixed}, we have $u(kh,t) = 0, \forall k \in \bbZ$ and $u(x,t)$ is even if $u_0(x)$ is even. Differentiating the equation with respect to $x$ gives
	\begin{align*}
		u_{tx} (x,t) + \left( u_x(x+h,t) - u_x(x-h,t) \right) u_x(x,t) + \left( u(x+h,t) - u(x-h,t) \right) u_{xx}(x,t) = 0.
	\end{align*}
Observe that $u_x(3kh,t)=0$ for all time by an argument similar to proof of Lemma \ref{boundarylemma}.
	Plugging in $x=h,2h$ gives
	\begin{align*}
		F_1'(t) + F_1(t) F_2(t) &= 0 \\
		F_2'(t) - F_1(t) F_2(t) &= 0,
	\end{align*}
	where $F_1(t) = u_x(h,t)$ and $F_2(t) = u_x(2h,t)$. Solving this system of ordinary differential equations gives
	\begin{align*}
		F_1(t) + F_2(t) &= F_{1}(0) + F_2(0) = A \\
		F_1'(t) &= F_1^2(t) - A F_1(t),
	\end{align*}
	for some constant $A$. The solution to this differential equation is
	\begin{align*}
		F_1(t) = \frac{A \exp{(A B)}}{\exp{(AB)}-\exp{(At)}},
	\end{align*}
	where $B = \frac{\ln F_1(0) - \ln(-F_2(0))}{F_1(0)+F_2(0)}$.
This blows up in finite time if $F_2(0) = u_x(2h,0) < 0$ and $B >0$.
\end{proof}
\begin{remark}
	For simplicity, take $h=4/3$. Then we can take $$u(x,0) = u_0(x) = \frac{16 (x-4)^2 (x+4)^2 x^2 (3x-8)(3x+8)(3x+4)(3x-4)}{3375(112 + 153x^2)}.$$ This satisfies our assumptions in Lemma \ref{specific IC} and thus blows up in finite time.
\end{remark}

\subsection{Non Blow Up} \label{Non Blow Up}
	We will now take specific initial data to \eqref{minus} and show that it cannot blow up in finite time. Let $u(x,t) = \sin\left(\frac{\pi x k}{h}\right),$ where $h$ is fixed and $k \in \bbZ$. Noting that $u_t = 0$ and $u(x+h,t) - u(x-h,t) = 0$ (by trigonometric identities), $u(x,t)$ solves the equation. We know $u(x,t) = \sin\left(\frac{\pi x k}{h}\right)$ never blows up.
	
	Similarly, for \eqref{plus}, we will take $u(x,t) = \sin \left( \frac{\pi x (k-1/2)}{h} \right)$, where $h$ is fixed and $k\in \bbZ$. Once again, noting that $u_t = 0$ and $u(x+h,t) + u(x-h,t) = 0$, $u(x,t)$ solves the equation. We also know that $u(x,t) =  \sin \left( \frac{\pi x (k-1/2)}{h} \right)$ never blows up. So we have found stationary solutions for both equations \eqref{plus} and \eqref{minus} that never blow up in finite time.
	So the nonlocal models are different from the Burgers equation where any nonconstant solution blows up in finite time: there exists non-trivial initial data for which solutions are globally regular for the nonlocal equation.
	
	We can also construct stationary solution to equation \eqref{minus} by setting the period $L$ to be $h$. The nonlocal terms become $u(x+h,t) = u(x-h,t) = u(x,t)$ so \eqref{minus} reduce to $u_t = 0$. This is constant in time. Therefore $u(x,t) = u_0 (x)$ for all $t$, so given a smooth initial condition, $u(x,t)$ will not blow up.

\section{Simulations} \label{Simulations}
In this section, we compare our model with the well known "local "Burgers equation \eqref{InviscidBurgers}. We used Matlab v2013 to run all simulations, with a forward in time, centered in space scheme. We illustrate many of the results of this paper in the graphics we generate.

We first look at the "local" Burgers equation, \eqref {InviscidBurgers}. We know that this leads to gradient catastrophe (i.e. blow up in gradient) in finite time for all non-constant smoooth initial data. We use $u(x,0) = \sin(\pi x)$ to generate Figure \ref{Burgers}.

\begin{figure}[H]
\centering
\begin{minipage}[h]{0.5\linewidth}
  \includegraphics[scale=0.2]{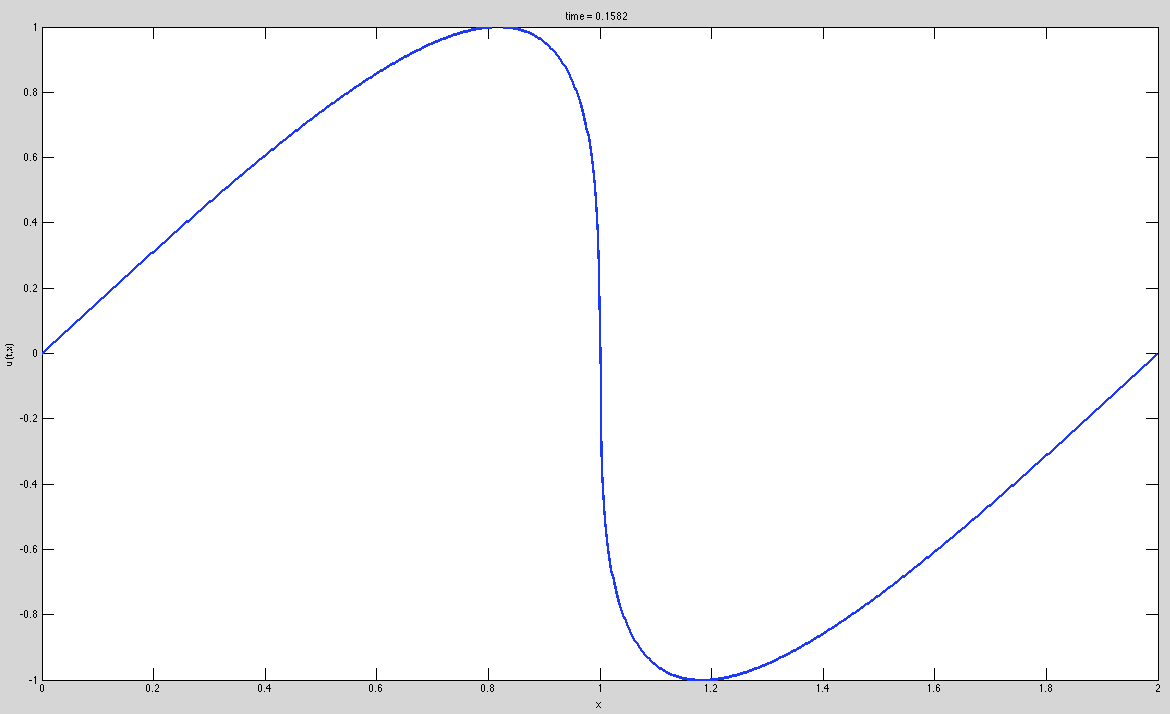}
  \caption{Local Burgers Equation, $h=0$}
  \label{Burgers}
\end{minipage}
\quad
\begin{minipage}[H]{0.45\linewidth}
  \includegraphics[scale=0.2]{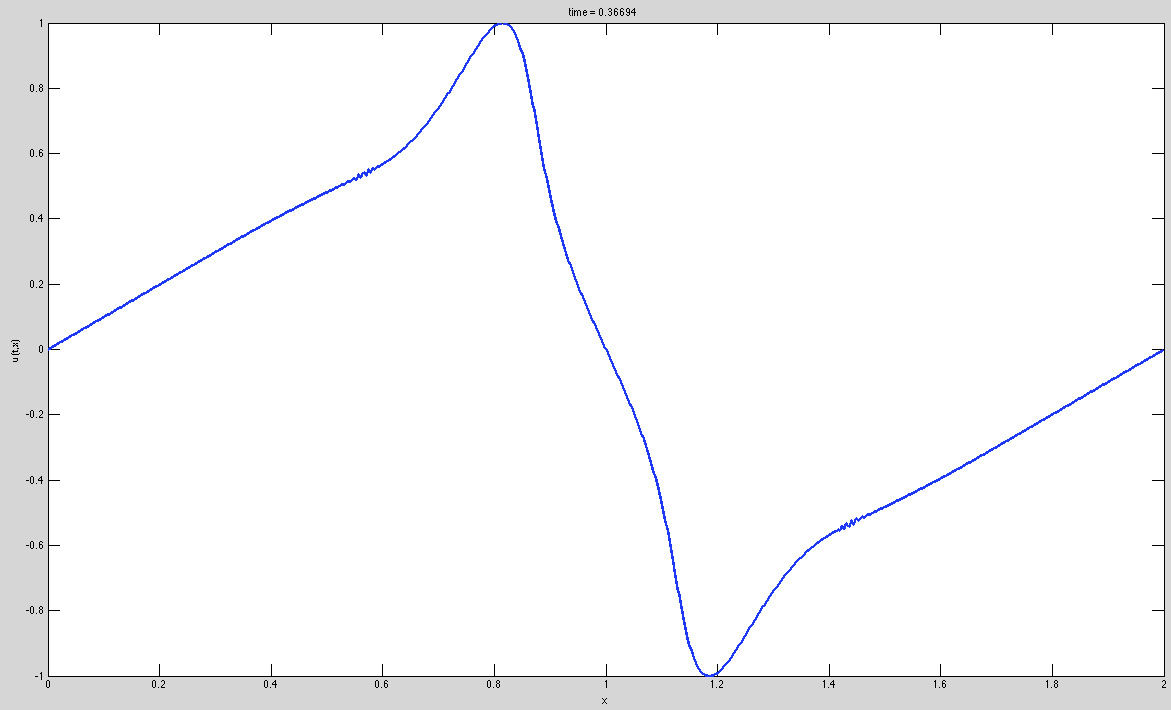}
  \caption{Nonlocal Burgers, $h = L/8$}
  \label{Nonlocal Burgers}
\end{minipage}
\end{figure}

As we can see, the slope of the graph in \ref{Burgers} at $x=0$ blows up in finite time. Now, considering our equation with the plus sign, $\InviscidBurgersNonlocalPlus{u}$, notice that there is a translation parameter $h$ in our equation which affects the location of blow up. As we can see in Figure \ref{Nonlocal Burgers}, with $h=L/8$ where $L$ is the period of the initial data, blow up does not occur at the origin, and two peaks form instead of the usual one.
Below, we varied the value of $h$ to be $h=L/16$ and $L/32$ in Figures \ref{L/16} and \ref{L/32}, respectively, which gives blow up closer and closer to the origin.

\begin{figure}[h]
\centering
\begin{minipage}[h!]{0.5\linewidth}
  \includegraphics[scale=0.2]{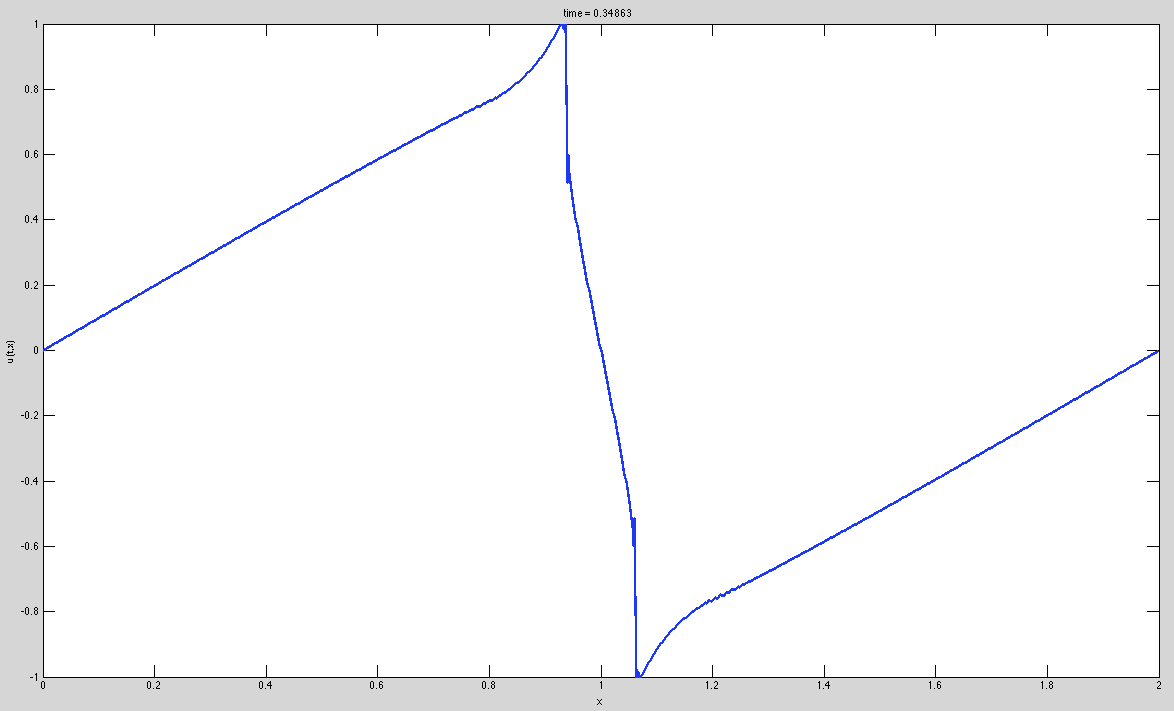}
  \caption{Nonlocal Burgers, $h=L/16$}
  \label{L/16}
\end{minipage}
\quad
\begin{minipage}[h]{0.45\linewidth}
  \includegraphics[scale=0.2]{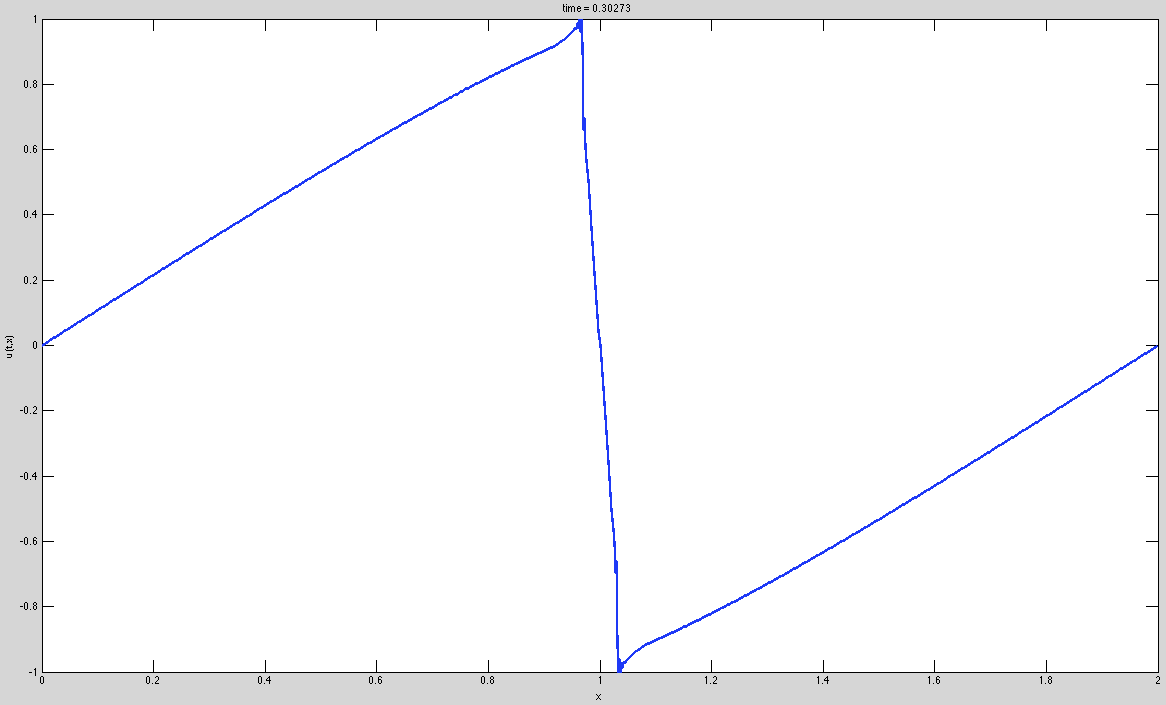}
  \caption{Nonlocal Burgers, $h=L/32$}
  \label{L/32}
\end{minipage}
\end{figure}

Now we constructed initial data to fit Lemma \ref{specific IC} to get intuition on how it will blow up at $x=\pm L/3, \pm 2L/3$ in the minus sign case. Figure \ref{IC} shows the initial data for our equation $\InviscidBurgersNonlocalMinus{u}$. Note how $u(x,0) = 0$ at $x=kh$, where period $L=6h$. Now in Figure \ref{blow up}, we see that at $x=\pm L/3, \pm 2L/3$, vertical lines form, causing blow up in slope.

\begin{figure}[h]
\centering
\begin{minipage}[h!]{0.5\linewidth}
  \includegraphics[scale=0.2]{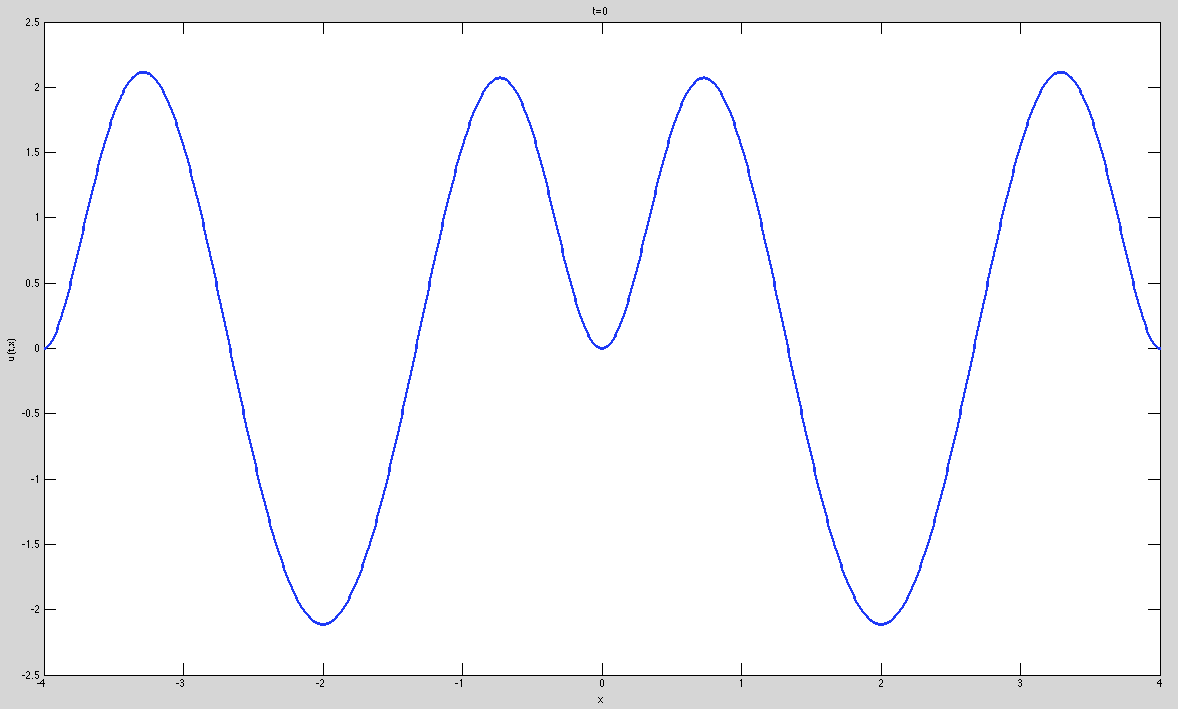}
  \caption{Nonlocal Burgers minus case initial condition}
  \label{IC}
\end{minipage}
\quad
\begin{minipage}[h]{0.45\linewidth}
  \includegraphics[scale=0.2]{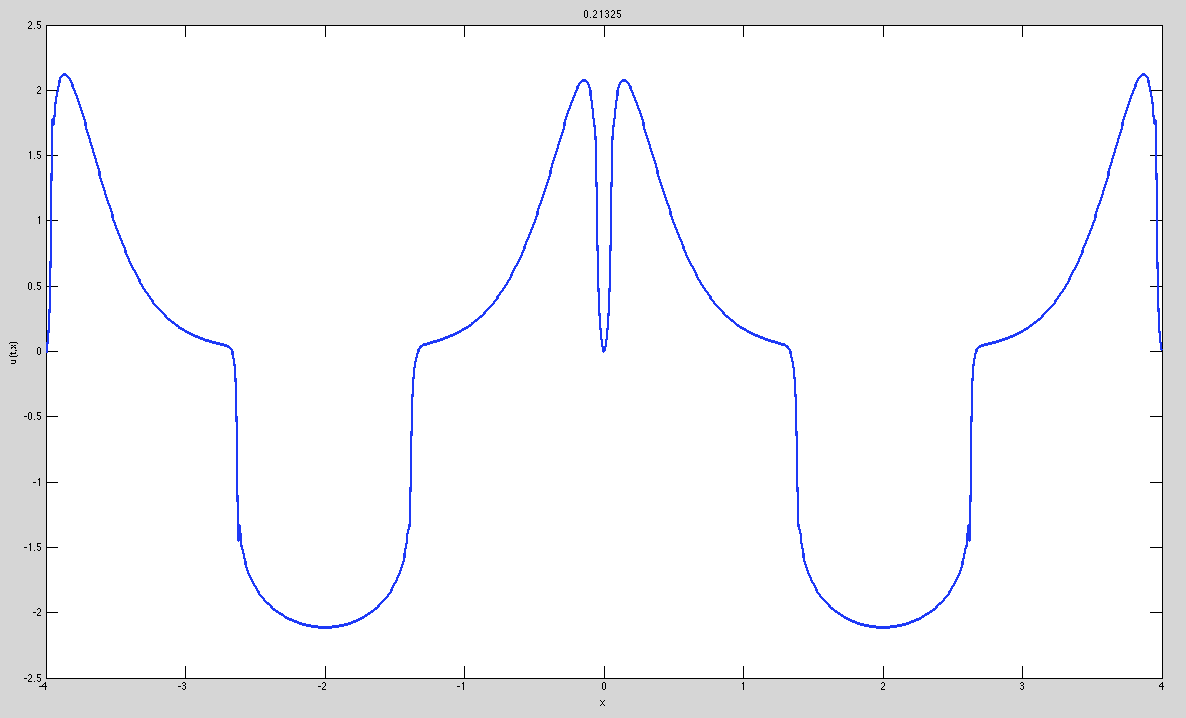}
  \caption{Nonlocal Burgers minus case in finite time}
  \label{blow up}
\end{minipage}
\end{figure}

\section*{Acknowledgement}
	Our research was done during the 2013 University of Wisconsin - Madison REU, sponsored by NSF grant RTG: Analysis and Applications 1147523. We would like to thank Professor Alexander Kiselev for his introduction of the topic and guidance throughout. We would also like to thank Kyudong Choi and Tam Do for their help with the many issues we ran in to.

\nocite{EvansPDE}
\bibliographystyle{alpha}
\bibliography{Bibliography}

\end{document}